\documentclass{amsart}
\usepackage{amssymb,latexsym}
\usepackage{amsmath}
\usepackage[dvipdfm]{graphicx}
 \usepackage{color}
\usepackage{enumerate}
\newenvironment{enumeratei}{\begin{enumerate}[\upshape (i)]}{\end{enumerate}}
\numberwithin{equation}{section}
\theoremstyle{plain}
 \newtheorem{theorem}{Theorem}[section]
 \newtheorem{lemma}[theorem]{Lemma}
 \newtheorem{proposition}[theorem]{Proposition}
 \newtheorem{corollary}[theorem]{Corollary}
\theoremstyle{definition}
 \newtheorem{definition}[theorem]{Definition}
 \newtheorem{remark}[theorem]{Remark}
 
 \newtheorem{example}[theorem]{Example}

 \newtheorem*{ackno}{Acknowledgment}
\newcommand \url [1] {{\tt{#1}}}

\newcommand \Clm {C_{\ell\kern-0.5pt\textup m}}
\newcommand \Crm {C_{\textup{rm}}}
\newcommand \Compp [1] {\textup{ComplP}(#1)}
\newcommand \Pcompp [1] {\textup{PseudCP}(#1)}
\newcommand \notmid{\mathrel{|\kern -3.8pt\mathord /}}
\newcommand \np {\widehat n}
\renewcommand \mp {\widehat m}
\newcommand \grid {\textup{Grid}}
\newcommand \sgrid {\textup{SGrid}} 
\newcommand \powset {\textup{PowSet}}
\newcommand \isyst {\mathcal H}
\newcommand \jsyst {\mathcal T}
\newcommand \ksyst {\mathcal W}
\newcommand \biz {\mathcal B}
\newcommand \cfunc [2] {f(#1,#2)}
\newcommand \board[2] {\textup{Board}(#1,#2)}
\newcommand\ofromto [3] { \textup{From}\kern-1pt\textup{To}_{#1}(#2,#3)} 
\newcommand\fromto [3] { \textup{HInt}_{#1}(#2,#3)} 

\newcommand \lat [1] {\textup{Lat}(#1)}
\newcommand \lmpt [1] {\textup{LPt}(#1)}
\newcommand \rmpt [1] {\textup{RPt}(#1)}
\newcommand \lowstar [1] {{#1}_\ast}
\newcommand \Jir [1] {\textup{Ji}(#1)} 
\newcommand \Mir [1] {\textup{Mi}(#1)}
\newcommand \atoms [1] {\textup{Atoms}(#1)}
\newcommand \length [1] {\textup{length}(#1)}

\newcommand \nablaell [1] {\nabla_{\kern -2pt #1}}

\newcommand \qq[1] {{\kern #1 pt}}

\newcommand\ideal[1]{\mathord\downarrow #1}

\newcommand \tuple [1] {\langle #1\rangle}
\newcommand \pair [2] {\tuple{#1,#2}}
\renewcommand\phi{\varphi}
\renewcommand\emptyset{\varnothing}

\newcommand \tbf [1] {\textbf{#1}} 
\newcommand \set[1] {\{#1\}}
\newcommand \bigset[1] {\bigl\{#1\bigr\}}
\renewcommand\phi{\varphi}
\renewcommand\epsilon{\varepsilon}
\newcommand \then {\mathrel{\Rightarrow}} 
 
\newcommand \nonparallel {\mathrel{
\not\mathord{\kern -1.5 pt\parallel}}}

\newcommand\init [1] {} 
\newcommand\nothing [1] {}

\begin{document}
\title[CD-independent subsets in meet-distributive lattices]
{CD-independent subsets in meet-distributive lattices}

\author[G.\ Cz\'edli]{G\'abor Cz\'edli}
\email{czedli@math.u-szeged.hu}
\urladdr{http://www.math.u-szeged.hu/\textasciitilde{}czedli/}
\address{University of Szeged, Bolyai Institute. 
Szeged, Aradi v\'ertan\'uk tere 1, HUNGARY 6720}

\thanks{This research was supported by
the European Union and co-funded by the European Social Fund  under the project ``Telemedicine-focused research activities on the field of Mathematics, 
Informatics and Medical sciences'' of project number    ``T\'AMOP-4.2.2.A-11/1/KONV-2012-0073'', and by  NFSR of Hungary (OTKA), grant number 
K83219}



\keywords{
CD-independent subset, laminar system, meet-distributive lattice, convex geometry of circles, number of islands}

\date{July 9, 2013\\
\phantom{nk} 2000 \emph{Mathematics Subject Classification}.  {Primary 06C10; secondary 05A05 and 05B25}
}

\begin{abstract} 
A subset $X$ of a finite lattice $L$ is CD-independent if the meet of any two incomparable elements of $X$ equals 0. In 2009, Cz\'edli, Hartmann and Schmidt 
proved that any two maximal CD-independent subsets of a finite distributive lattice 
have the same number of elements. In this paper, we prove that if $L$ is a finite meet-distributive lattice, then the size of every CD-independent subset of $L$ is at most the number of atoms of $L$ plus the length of $L$. If, in addition, there is no three-element antichain of meet-irreducible elements, then we give a recursive description of maximal CD-independent subsets. Finally, to give an application of CD-independent subsets, we give a new approach to count islands on a rectangular board.
\end{abstract}

\maketitle

\section{Introduction and the main result}
\label{introsection}
\subsection{Outline and goals}
The concept of CD-independent subsets in lattices was introduced in Cz\'edli, Hartmann and Schmidt~\cite{czg-h-sch}. 
The primary purpose of the present paper is to generalize the main result of \cite{czg-h-sch} from distributive lattices to meet-distributive ones. The secondary goal is to give a combinatorial application of the lattice-theoretical paper \cite{czg-h-sch} by counting islands. 
Since the  cross-reference between the combinatorial part, 
Section~\ref{applicationssection}, 
and the lattice theoretical part, Sections~\ref{introsection}--\ref{examplessection}, is minimal, see  Definition~\ref{cddef} and Proposition~\ref{czghschprop}, readers interested mainly in combinatorics can start directly with Section~\ref{applicationssection}. 

After recalling some lattice theoretical concepts, the present section formulates the main result, Theorem~\ref{thmmain}. Its proof is given in Section~\ref{circlessection}, while Section~\ref{examplessection} gives some examples that  rule out certain generalizations.
 
\subsection{Basic concepts from Lattice Theory} 
All lattices in the present paper are assumed to be finite, even if this is not repeated all the times.
For $u\neq 0$ in a (finite) lattice $L$, let $\lowstar u$ denote the meet of all lower covers of $u$. If the interval $[\lowstar u,u]$ is a distributive lattice for every $u\in L\setminus \set0$, then $L$ is \emph{meet-distributive}. This concept goes back to 
\init{R.\,P.\ }Dilworth \cite{r:dilworth40} but there are more than a dozen equivalent definitions. In fact, meet-distributivity or its dual is one of the most often rediscovered concepts in Lattice Theory; see \init{K.\ }Adaricheva~\cite{adaricheva},  \init{K.\ }Adaricheva, \init{V.\,A.\ }Gorbunov and \init{V.\,I.\ }Tumanov~\cite{r:adarichevaetal}, \init{B.\ }Monjardet~\cite{monjardet},   and 
\init{N.\ }Caspard and 
\init{B.\ }Monjardet~\cite{caspardmonjardet}; see also  \init{G.\ }Cz\'edli~\cite[Proposition 2.1 and Remark 2.2]{czgcoord} and \init{K.\ }Adaricheva and \init{G.\ }Cz\'edli \cite{adarichevaczg} for recent surveys.

As usual, a  finite lattice $L$ is \emph{lower semimodular} if whenever $a,b\in L$ such that $a$ is covered by $a\vee b$, in notation $a\prec a\vee b$, then $a\wedge b\prec b$. Equivalently, if the implication $a\prec b\then a\wedge c\preceq b\wedge c$ holds for all $a,b,c\in L$. We will often use the fact, without further reference, that 
finite meet-distributive lattices are lower semimodular; see \init{R.\,P.\ }Dilworth \cite{r:dilworth40} and \init{B.\ }Monjardet~\cite{caspardmonjardet}, or see also the dual of  \init{G.\ }Cz\'edli~\cite[Proposition 2.1 and Remark 2.2]{czgcoord} for an overview. 

An element of $L$ is \emph{meet-irreducible} if it has exactly one cover. The set of meet-irreducible elements of $L$ is denoted by  $\Mir L$. The set $\Jir L$ of \emph{join-irreducible} elements is defined dually.
Following \init{G.\ }Gr\"atzer and \init{E.\ }Knapp \cite{gratzerknapp} and, in the present form,  \init{G.\ }Cz\'edli and \init{E.\,T.\ }Schmidt \cite{czgschtJH}, $L$ is \emph{dually slim} if $\Mir L$ contains no three-element antichain. Due to \init{G.\ }Cz\'edli~\cite{czgcircles}, and to the dual of results in 
\init{G.\ }Cz\'edli and \init{G.\ }Gr\"atzer \cite{czgggresect} and 
 \init{G.\ }Cz\'edli and \init{E.\,T.\ }Schmidt \cite{czgschtJH}, \cite{czgschvisual}, and \cite{czgschslim2}, dually slim lower semimodular lattices are understood quite well.

\begin{remark}\label{remsmdmlyn} It follows from the dual of Cz\'edli,  Ozsv\'art, and Udvari~\cite{czgozsvudv} and Cz\'edli and Schmidt~\cite[Corollary 3.5]{czgschcompser} 
that, for a finite lattice $L=\tuple{L;\leq}$, the following three conditions are equivalent:
\begin{itemize}
\item $L$ is meet-distributive and dually slim;
\item $L$ is lower semimodular and dually slim;
\item there exist an $n\in\mathbb N_0$, a finite group $G$, and composition series $\set1=H_0\triangleleft H_1\triangleleft  \dots \triangleleft H_n=G$ and $\set1=K_0\triangleleft K_1\triangleleft  \dots \triangleleft K_n=G$ such that $\tuple{L;\leq}$ is isomorphic to $\tuple{\set{H_i\cap K_j: 0\leq i,j\leq n};\subseteq}$.
\end{itemize}
\end{remark}

For a lattice $L$, let $\atoms L$ and $\length L$ stand for the set of atoms of $L$ and the length of $L$, respectively. Since we only deal with lower semimodular, finite lattices, $\length L$ equals the size of any maximal chain minus 1.  

\begin{definition}[Cz\'edli, Hartmann and Schmidt~\cite{czg-h-sch}]\label{cddef}
A subset $X$ of a lattice $L$ is \emph{CD-independent} if for any $x,y\in X$ such that $x$ and $y$ are incomparable (in notation, $x\parallel y$), we have $x\wedge y=0$. In other words, if any two elements of
$X$ are either Comparable, or Disjoint; this is were the acronym CD comes from. Note that CD-independence is also known as \emph{laminarity}, see Pach, Pluh\'ar,  Pongr\'acz, and Szab\'o~\cite{pachetall}.
\end{definition}

The main result of \cite{czg-h-sch} is the following:
\begin{proposition}[\cite{czg-h-sch}]\label{czghschprop} Let $L$, $X$, and $C$ be a finite, lower semimodular lattice, a maximal CD-independent subset of $L$, and a maximal chain of $L$, respectively. Then the following two assertions hold.
\begin{itemize}
\item $C\cup\atoms L$ is a maximal CD-independent subset of $L$.
\item If, in addition, $L$ is distributive, then $|X|=|C\cup\atoms L|$, that is, the size of every maximal CD-independent subset is $\length L + |\atoms L|$.
\end{itemize}
\end{proposition}

Note that it is also possible to define the concept of CD-independent subsets of posets, see Horv\'ath and Radeleczki~\cite{horvathradCD}, but the present paper is restricted to lattices. 
In view of further results in \cite{czg-h-sch}, we cannot expect that the second part of this proposition extends to a significantly larger class of lattices. However, replacing distributivity by meet-distributivity, which is a weaker assumption, the theorem below still shows some property of CD-independent subsets. 

For a poset $H$, let $\max(H)$ stand for the set of maximal elements of $H$. If $u$ is an element of a lattice $L$, then the principal ideal $\set{x\in L: x\leq u}$ is denoted by $\ideal u$. For $a,b\in L$, $\pair ab$ is a \emph{complemented pair} if $a\wedge b=0$ and $a\vee b=1$. Given $a\in L$, if $c$ is the largest element of $L$ such that $a\wedge c=0$, then $c$ is the \emph{pseudocomplement} of $a$. Note that $a$ need not have a pseudocomplement, but if it has, then its pseudocomplement is uniquely determined.   If $a$ and $b$ are mutually pseudocomplements of each other, then $\pair ab$ is a \emph{pseudocomplemented pair}.

Since the concept of pseudocomplemented pairs seems to be new, some comments are appropriate here. The five-element nonmodular lattice, $N_5$, witnesses that a complemented pair need not be a pseudocomplemented pair. The lattice obtained from $N_5$ by adding a new top element shows that a pseudocomplemented pair need not be a complemented pair. 
However, 
\begin{equation}\label{dstrlcppscp}
\begin{aligned}
&\text{if $L$ is a distributive lattice, then every complemented} \cr
&\text{pair $\pair ab$ of $L$ is a pseudocomplemented pair}
\end{aligned}
\end{equation}
since if $a\wedge x=0$, then $x=1\wedge x=(a\vee b)\wedge x=(a\wedge x)\vee (b\wedge x)=0\vee (b\wedge x)=b\wedge x$ implies $x\leq b$, whence $b$ is the pseudocomplement of $a$.

\subsection{Main result}
For a subset $X$ of $L$, we let
\[
\begin{aligned}
\Compp L&=\set{\pair ab\in L^2: \pair ab\text{ is a complemented pair of }L},\cr
\Pcompp L&=\set{\pair ab\in L^2: \pair ab\text{ is a pseudocomplemented pair of }L}\text.
\end{aligned}
\]
Since  $\pair 01\in \Pcompp L\cap\Compp L$ but we always want to exclude this possibility, we shall stipulate  $a\parallel b$ or $\set{a,b}\subseteq \Mir L$. Our main goal is to prove the following result. While its Part~\eqref{parte} is relatively simple, Parts~\eqref{partk} and \eqref{parth} give a lot of additional information on maximal CD-independent subsets for two important subclasses of meet-distributive lattices. Obviously, if $X$ is a \emph{maximal} CD-independent subset of $L$, then $\set{0,1}\cup\atoms L\subseteq X$. Therefore, to characterize maximal CD-independent subsets in the theorem below, it suffices to consider those subsets $X$ of $L$ that extend $\set{0,1}\cup\atoms L$.

\begin{theorem}\label{thmmain} 
Let $L$ be a finite lattice consisting of at least two elements.
\begin{enumerate}[\upshape (1)]
 \item\label{parte} If $L$ is    meet-distributive and $Y$ is a CD-independent subset of $L$, then we have $|Y|\leq \length L + |\atoms L|$.
\end{enumerate}
In the rest of the theorem, let $X$ be a subset of $L$ such that $\set{0,1}\cup\atoms L\subseteq X$. We denote $|\!\max(X\setminus\set 1)|$ by $k$, and we let $\max(X\setminus\set 1)=\set{a_1,\dots, a_k}$.
\begin{enumerate}[\upshape (1)]\setcounter{enumi}{1}
 \item\label{partk} If $L$ is dually slim and lower semimodular $($equivalently, dually slim and meet-distributive$)$, then following three conditions are equivalent.
  \begin{enumerate}[\upshape (a)]
   \item\label{partka} $X$ is a maximal CD-independent subset of $L$.
   \item\label{partkb} 
    \begin{enumerate}[\upshape (i)]
     \item\label{partkba} Either $k=1$, $a_1$ is a coatom of $L$, and $X\setminus\set 1$ is a maximal CD-independent subset of $\ideal{a_1}$, 
     \item\label{partkbb} or $k=2$, $\pair{a_1}{a_2}\in \Compp L\cap \Pcompp L$,  $\set{a_1,a_2}\subseteq \Mir L$, and 
        $X\cap\ideal{a_i}$ is a maximal CD-independent subset of $\ideal{a_i}$ for $i=1,2$.
    \end{enumerate}
   \item\label{partkc} Either \eqref{partkba} above, or \eqref{partkbb} above with replacing  ``$\set{a_1,a_2}\subseteq \Mir L$'' by the condition ``$\,a_1\parallel a_2$''.
  \end{enumerate}
 \item\label{parth} If $L$ is distributive, then the following two conditions are equivalent.
  \begin{enumerate}[\upshape (a)]
   \item\label{partha} $X$ is a maximal CD-independent subset of $L$.
   \item\label{parthb} 
    \begin{enumerate}[\upshape (i)]
     \item\label{parthba}  Either $k=1$, $a_1$ is a coatom of $L$, and $X\setminus\set 1$ is a maximal CD-independent subset of $\ideal{a_1}$, 
     \item\label{parthbb} or $k=2$, $\pair{a_1}{a_2}\in \Compp L$, $a_1\parallel a_2$, and $X\cap\ideal{a_i}$ is a maximal CD-independent subset 
       of $\ideal{a_i}$ for $i=1,2$.
    \end{enumerate}
  \end{enumerate}
\end{enumerate}
\end{theorem}

\begin{remark} By Remark~\ref{remsmdmlyn}  and Proposition~\ref{czghschprop}, $C\cup\atoms L$ is always a maximal CD-indep\-endent subset, provided $L$ is meet-distributive and $C$ is a maximal chain of $L$. Furthermore, $|C\cup\atoms L|=\length L+|\atoms L|$ by lower semimodularity. Hence, the upper bound in Part \eqref{parte} is sharp. Note also that Parts~\eqref{partk} and \eqref{parth} give recursive descriptions for maximal CD-independent subsets.
\end{remark} 

The following statement will be derived from Part~\eqref{partk} of Theorem~\ref{thmmain}.

\begin{corollary}\label{corpsccmirr}
If $L$ is a finite, dually slim, lower semimodular lattice, $a_1,a_2\in L\setminus\set{0,1}$, and  $\pair{a_1}{a_2}\in \Compp L\cap \Pcompp L$, then $a_1,a_2\in\Mir L$.
\end{corollary}

In view of the fact that the analogous statement fails in the eight-element boolean lattice, this corollary is a bit surprising.

%

\section{Circles and the proof of the main result}\label{circlessection}
Before proving Theorem~\ref{thmmain}, we recall some results from \init{G.\ }Cz\'edli \cite{czgcircles}. Note that this will be the first application of the main result of \cite{czgcircles}.
As usual,  a \emph{circle} in the plane is a set $\set{\pair xy: (x-u)^2+(y-v)^2=r^2}$ where $u,v,r\in \mathbb R$ and $r\geq 0$. Let $F$ be a finite set of circles in the plane. A subset $Y$ of $F$ is \emph{closed} if whenever $C\in F$ and $C$ is in the convex hull of $\bigcup\set{D: D\in Y}$, then $C\in Y$. Less formally (but not quite precisely), if $Y$ is closed with respect to the usual convex hull operation restricted to $F$. 
Let $\lat F$ denote the set of closed subsets of $F$. With respect to inclusion, $\lat F$ is a lattice, and $\emptyset, F\in \lat F$. We call $\lat F$ a \emph{lattice of circles}.
If the centers of the circles in $F$ are on the same line, then $F$ is \emph{collinear}. In the collinear case, we always assume that the line containing the centers is the $x$ axis. A collinear set $F$ of circles is \emph{separated} if no point of the $x$ axis belongs to more than one member of $F$. For example, if we disregard the dotted arcs, then  $F$ depicted in Figure~\ref{fig1} is a separated collinear set of circles.

\begin{figure}[htc]
\centerline
{\includegraphics[scale=1.0]{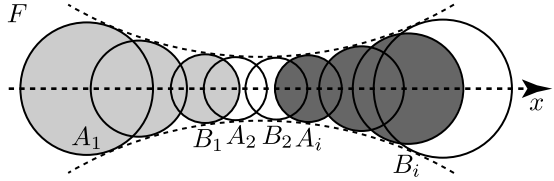}}
\caption{A separated concave set of collinear circles 
\label{fig1}}
\end{figure}

Next, let $F$ be a separated collinear set of circles. For $C\in F$, $C$ is of the form $\set{\pair xy: (x-u)^2+y^2=r^2}$, where $0\leq r\in\mathbb R$. The points $\lmpt C=\pair{u-r}0$ and $\rmpt C=\pair{u+r}0$ on the $x$-axis are the \emph{leftmost point} and the \emph{rightmost point} of $C$, respectively.
Note that for points  $\pair a0$ and $\pair b0$ on the $x$-axis, $\pair a0\leq\pair b0$ is always understood as $a\leq b$. 
For $A,B\in F$, let
\[\ofromto FAB=\set{C\in F: \lmpt A\leq \lmpt C\,\text{ and }\, \rmpt C\leq \rmpt B}\text.
\]
This set can be empty, and even if it is not empty, neither $A$, nor $B$ has to belong to it. For example, if $B$ is encapsulated in $A$, then $\ofromto FAB$ contains $B$ but not $A$. If $A,B\in \ofromto FAB$, then we write $\fromto FAB$ instead of $\ofromto FAB$, and we call $\fromto FAB$ a \emph{horizontal interval} determined by $A$ and $B$. That is,
\[
\begin{aligned}
\fromto FAB&=\set{C\in F: \lmpt A\leq \lmpt C\,\text{ and }\, \rmpt C\leq \rmpt B},\cr
&\text{provided }\lmpt A \leq \lmpt B\text{ and }\rmpt A \leq \rmpt B\text.
\end{aligned}
\]
For example, $\fromto F{A_1}{B_1}$  in Figure~\ref{fig1} consists of the light grey circles while $\fromto F{A_i}{B_i}$ from the dark grey ones. Note that, for a pair $\pair AB\in F^2$, the horizontal interval $\fromto FAB$ is not necessarily defined. 
If we want to emphasize that $\fromto FAB$ exists, then we write $\ofromto FAB=\fromto FAB$. 
We say that $F$ is a \emph{concave set} of collinear circles if for all $C_1,C_2,C_3\in F$, the conjunction of $\lmpt{C_1}\leq \lmpt{C_2}$ and $\rmpt{C_2}\leq \rmpt{C_3}$ implies that the smallest closed subset of $F$ that contains $C_1$ and $C_3$ also contains $C_2$. In other words, $F$ is concave if for all $X\in \lat F$ and $A,B\in X$ such that $\fromto FAB$ defined, $\fromto FAB\subseteq X$. 
In Figure~\ref{fig1}, the circles determined by dotted arcs do not belong to $F$; their purpose is to indicate what concavity means. However, this figure does not reflect generality since $F$ in Figure~\ref{fig2} with encapsulated circles is also a concave set of collinear circles.

Using a result of Edelman~\cite{edelman}, see also \cite[Lemma 3.5]{czgcircles}, one can translate some results of \cite{czgcircles} to the language of Lattice Theory as follows.

\begin{proposition}[{\cite[Proposition 2.1, Theorem 2.2, and Lemma 3.1 ]{czgcircles}}]\label{cirlProp}\ 
\begin{enumeratei}
\item[\textup{(A)}]
If $F$ is a finite set of circles in the plane, then $\lat F$ is a meet-distributive lattice. 
\item[\textup{(B)}]
Dually slim, lower semimodular lattices are, up to isomorphism, characterized as lattices $\lat F$ where $F$ is a separated concave set of collinear circles.
\item[\textup{(C)}] If  $F$ is a separated concave set of collinear circles, then we have
$\lat F=\set{\emptyset}\cup\{\fromto FAB:A,B\in F\}$. Furthermore, for each $\emptyset\neq X\in\lat F$, there are a unique $A\in X$ and a unique $B\in X$ such that $X=\fromto FAB$.
\end{enumeratei}
\end{proposition}

Now, we are ready to prove our main result.

\begin{proof}[Proof of Theorem~\ref{thmmain}]
We prove Part \eqref{parte} by induction on $|L|$. For $|L|\leq 4$, $L$ is distributive and \eqref{parte} follows from Proposition~\ref{czghschprop}. 

Assume that $|L|>4$ and that Part \eqref{parte} holds for all lattices of smaller size. 
Let $k=|\max(Y\setminus\set 1)|$, and let $\max(Y\setminus\set 1)=\set{a_1,\dots, a_k}$. We can assume that $Y$ is a maximal CD-independent subset of $L$; this assumption implies $\atoms L\subseteq Y$.

Assume first that $k=1$, and let $Y_1=Y\cap\ideal {a_1}=Y\setminus\set1$. 
Clearly, $Y_1$ is a CD-independent set of $\ideal{a_1}$. By the induction hypothesis, $|Y_1|\leq \length{\ideal{a_1}}+|\atoms{\ideal{a_1}}|$. This, together with 
$\atoms{\ideal{a_1}}\subseteq \atoms L$ and 
$1+\length{\ideal{a_1}}\leq \length L$,  yields 
\[|Y|=1+|Y_1| \leq 1+\length{\ideal{a_1}}+|\atoms{\ideal{a_1}}|\leq \length L+ |\atoms L|\text,
\]   
as desired.

Next, we assume $k\geq 2$. 
Since $\atoms L\subseteq Y\setminus\set 1 \subseteq \ideal{a_1}\cup\dots\cup \ideal {a_k}$, we  conclude
$\atoms L=\atoms{\ideal{a_1}}\cup\dots\cup\atoms{\ideal {a_k}}$.  Here the union is disjoint since $\set{a_1,\dots,a_k}$ is an antichain and $a_i\wedge a_j=0$ for $i\neq j$ by the CD-independence of $Y$. Consequently,
\begin{equation}\label{eqaduni}
|\atoms L|=|\atoms{\ideal{a_1}}|+\dots+|\atoms{\ideal{a_k}}|\text.
\end{equation}
For $1\leq i<j\leq k$, we have $\Jir{\ideal{a_i}} \cap \Jir{\ideal{a_j}}=\emptyset$ since $a_i\wedge a_j=0$. On the other hand, $\Jir{\ideal{a_t}}\subseteq \Jir L$ for $t=1,\dots, k$. Hence, $|\Jir{\ideal{a_1}}|+\dots+ |\Jir{\ideal{a_k}}|\leq|\Jir L|$. We know from 
\init{M.\ }Stern \cite[Theorem 7.2.27]{stern}, who attributes it to \init{S.\,P.\ }Avann~\cite{avann61} and \cite{avann64}, 
that $|\Jir K|=\length K$ for every meet-distributive lattice;  see also the dual of \init{G.\ }Cz\'edli~\cite[Proposition 2.1(iii)$\Leftrightarrow$(v)]{czgcoord} for more historical comments. Clearly, the ideals $\ideal{a_i}$ are meet-distributive. Thus the last inequality turns into
\begin{equation}\label{eqLhLn}
\length{\ideal{a_1}}+\dots+\length{\ideal{a_k}} \leq \length L \text.
\end{equation}

Next, for $i=1,\dots, k$, let $Y_i=Y\cap \ideal{a_i}$. Since $Y_i$ is clearly a CD-independent subset of $\ideal{a_i}$, the induction hypothesis gives
\begin{equation}\label{eqLkTn}
|Y_i|\leq |\atoms{\ideal{a_i}}| +  \length{\ideal{a_i}}\, \text{ for }i=1,\dots,k\text.
\end{equation}
Now, using the previous formulas, $Y_i\cap Y_j=\set{0}$ for $i\neq j$, and $k\geq 2$, we can compute as follows; note that $2$ at the beginning counts $0=0_L$ and $1=1_L$.
\begin{align*}
|Y|&=2+\sum_{i=1}^k |Y_i\setminus\set 0| 
\overset{\eqref{eqLkTn}}{\leq} 
2+\sum_{i=1}^k \bigl( |\atoms{\ideal{a_i}}| +  \length{\ideal{a_i}}-1 \bigr)\cr
&\leq \sum_{i=1}^k  |\atoms{\ideal{a_i}}| 
+ \sum_{i=1}^k  \length{\ideal{a_i}} 
\overset{\eqref{eqaduni},\eqref{eqLhLn} } {\leq}|\atoms L| + \length L \text.
\end{align*}
This completes the induction step and proves Part \eqref{parte}.

Next, we prove Part \eqref{partk}. Proposition~\ref{cirlProp}(B), together with Remark~\ref{remsmdmlyn}, allows us to assume that $L=\lat F$, where $F$ is a separated concave set of collinear circles.
Since $F$ is separated, it contains a unique \emph{leftmost circle}, $\Clm$. That is, we have $\lmpt\Clm<\lmpt D$ for all $D\in F\setminus\set{\Clm}$. Similarly, we have a unique \emph{rightmost circle} $\Crm\in F$ with the property $\rmpt D  <\rmpt\Crm$ for all $D\in F\setminus\set{\Crm}$. 

Assume Part~\eqref{partka}, that is, let $X$ be a maximal CD-independent subset of $L=\lat F$. 
By Proposition~\ref{cirlProp}(C), there exist unique $A_j$ and $B_j$ in $F$  such that we have $a_j=\fromto F{A_j}{B_j}$, for $j\in\set{1,\dots,k}$. 
For example, in Figure~\ref{fig1}, where the label of a circle is always below its center,  $a_1$ consists of the light grey circles while $a_i$ from the dark grey ones. 
Since $\lmpt {A_j}\leq\rmpt{A_j}$, 
$\lmpt {B_j}\leq\rmpt{B_j}$, and  $A_j, B_j\in a_j=\fromto F{A_j}{B_j}$, we know that
\begin{equation}\label{siTgj} \lmpt{A_j}\leq \rmpt{A_j}\leq \rmpt{B_j}\text{ and }\lmpt{A_j}\leq\lmpt{B_j}\leq\rmpt{B_j},
\end{equation} 
for $j\in\set{1,\dots,k}$. However, note that $\lmpt{B_j}\leq \rmpt{A_j}$ or even $A_j=B_j$ can occur. 

We assert that, for $i\neq j$, $\lmpt{A_i}\neq \lmpt{A_j}$ and $\rmpt{B_i}\neq \rmpt{B_j}$. 
(Since $F$ is separated, this is equivalent to the statement that the $A_i$ are pairwise distinct, and so are the $B_i$.)
By left-right symmetry, it suffices to deal with $\lmpt{A_i}$ and $\lmpt{A_j}$. 
Aiming for a contradiction, suppose  
$\lmpt{A_i} = \lmpt{A_j}$. If
$\rmpt{A_i}\leq \rmpt{A_j}$, then 
\[\lmpt{A_j}\leq\lmpt{A_i}\text{ and }
\rmpt{A_i}\leq  \rmpt{A_j}
\overset{\eqref{siTgj}}\leq \rmpt{B_j}
\]
gives $A_i\in a_i\cap a_j=\emptyset$, a contradiction. Otherwise, if $\rmpt{A_i} > \rmpt{A_j}$, we similarly obtain $A_j\in a_i\cap a_j=\emptyset$, a contradiction again. 
Thus the $\lmpt{A_i}$, $i\in \set{1,\dots,k}$, are pairwise distinct, and so are the $\rmpt{B_i}$. Hence, we can choose the indices such that 
\begin{equation}\label{dkHsV}
\lmpt{A_1}<\lmpt{A_2}<\dots<\lmpt{A_k}\text.
\end{equation}
If we had $\rmpt{B_i}\geq \rmpt{B_j}$ for some $1\leq i<j\leq k$, then  
\[\lmpt{A_i}\overset{\eqref{dkHsV}}<\lmpt{A_j}\text{ and }\rmpt{A_j}\overset{\eqref{siTgj}}{\leq}  \rmpt{B_j}\leq \rmpt{B_i}
\]
would give $A_j\in a_i\cap a_j$, a contradiction. This shows 
\begin{equation}\label{dkRsV}
\rmpt{B_1}<\rmpt{B_2}<\dots<\rmpt{B_k}\text.
\end{equation}

Next, for the sake of contradiction, suppose $k\geq 3$. Now,
\eqref{siTgj}, \eqref{dkHsV}, \eqref{dkRsV}, and $k\geq 3$ easily imply $a_1\vee a_2\subseteq \fromto F{A_1}{B_2}$, $B_2\notin a_1$, $A_1\notin a_2$,  and $B_k\notin \fromto F{A_1}{B_2}$. 
Hence,  $a_1\vee a_2$ is neither the largest element of $\lat F$, nor it is one of $a_1$ and $a_2$. Therefore, $X$ is a proper subset of $X'=X\cup\set{a_1\vee a_2}$. To obtain a contradiction, it suffices to prove that $X'$ is CD-independent. It is sufficient to show that $(a_1\vee a_2)\wedge a_i=0$ for $i\leq 3$ since $X$ is CD-independent.  Hence, for $i\geq 3$, it suffices to prove $\fromto F{A_1}{B_2}\cap \fromto F{A_i}{B_i}=\emptyset$.  Suppose for (an encapsulated) contradiction that a circle $C\in F$ belongs to this intersection. This gives
\[\lmpt{A_2}\overset{\eqref{dkHsV}}<\lmpt{A_i}\leq \lmpt C
\text{ and }
\rmpt C\leq \rmpt {B_2},
\]
implying  $C\in \fromto F{A_2}{B_2}\cap \fromto F{A_i}{B_i}=a_2\cap a_i= \emptyset$. This is a contradiction proving that $X'$ is CD-independent. However, this is impossible since $X$ was a maximal CD-independent subset of $\lat F$. This proves $k\leq 2$.

Armed with $k\leq 2$, first we deal with the case $k=2$. Since $\max(X\setminus\set 1)$ is an antichain, $a_1\parallel a_2$. Their meet is $0_L$, that is, $a_1\cap a_2=\emptyset$, since $X$ is CD-independent.
If we had $a_1\vee a_2<1$, then $X\cup \set{a_1\vee a_2}$ would also be CD-independent, in contradiction with the maximality of $X$.  Hence, $\pair{a_1}{a_2}\in \Compp L$.

Suppose for contradiction that $\lmpt\Clm<\lmpt{A_1}$. 
It follows from \eqref{siTgj}, \eqref{dkHsV}, and \eqref{dkRsV} that $\ofromto F{A_1}{B_2}=\fromto F{A_1}{B_2}$,  $a_1,a_2\subseteq \fromto F{A_1}{B_2}$, but ${\Clm}\notin\fromto F{A_1}{B_2}$, contradicting $a_1\vee a_2=1_L$.   
This proves the first half of
\begin{equation}\label{ksiszeig}
\begin{aligned}
a_1&=\fromto F{A_1}{B_1} =\fromto F{\Clm}{B_1} \text{ and, in particular,   }A_1=\Clm\cr
 a_2&=\fromto F{A_2}{B_2}=\fromto F{A_2}{\Crm}\text{ and, in particular,   }B_2=\Crm;
\end{aligned}
\end{equation}
its second half follows similarly.

If we had an $x\in L$ such that $a_2<x<1$ and $a_1\wedge x=0$, then $X\cup\set{x}$ would be a CD-independent subset that is strictly larger than $X$. Thus, we conclude that
\begin{equation}\label{fwhdez}
a_1\wedge a_2=0\,\text{ but, for every $\,x\in L$, }\,   a_2 < x < 1\text{ implies } a_1\wedge x\neq 0\text.
\end{equation}

Next, we show that, for every $V\in F$,
\begin{equation}\label{sisdHg}
\text{if  $a_1\cap \fromto FVV=\emptyset\,$, then $a_1\cap \fromto FV{\Crm}=\emptyset$.}
\end{equation}
Suppose the contrary, that is, let a circle $D\in F$ belong to $a_1\cap \fromto FV{\Crm}$ such $a_1\cap \fromto FVV=\emptyset$.
Since $V\in \fromto FVV$, we have $V\notin a_1=\fromto F{\Clm}{B_1}$, which gives 
$\rmpt{B_1} < \rmpt{V}$. 
Since we also have $\rmpt D\leq \rmpt{B_1}$ by $D\in a_1$, we conclude $\rmpt D\leq \rmpt V$ by transitivity. On the other hand, $D\in \fromto FV{\Crm}$ yields $\lmpt V\leq \lmpt D$, and it follows that $D\in \fromto FVV$. Hence, 
$D\in a_1\cap \fromto FVV$, which is a contradiction proving \eqref{sisdHg}.

Now, we are in the position to prove that $a_2$ is a pseudocomplement of $a_1$. Assume that $x\in L\setminus\set0$ such that $a_1\wedge x=0$. Consider an arbitrary circle $V$ in $x$. The obvious inequality $\fromto FVV\leq x$ implies $a_1\cap \fromto FVV=\emptyset$. Applying \eqref{sisdHg}, we obtain $a_1\cap \fromto FV{\Crm}=\emptyset$.
By  \eqref{fwhdez}, this rules out the inequality $a_2 < \fromto FV{\Crm}$, which is equivalent to $\lmpt V< \lmpt{A_2}$. Therefore,   $\lmpt{A_2}\leq\lmpt V$, and we have $V\in a_2$. Since $V\in x$ was arbitrary, we conclude $x\leq a_2$. This proves that $a_2$ is a pseudocomplement of $a_1$. An analogous argument shows that $a_1$ is a pseudocomplement of $a_2$. Thus $\pair{a_1}{a_2}\in \Pcompp L$.

To prove $a_1\in \Mir L$, assume $a_1=u_1\wedge u_2$. We have $u_1=\fromto F{U_i}{W_i}$ for $i=1,2$ and appropriate circles $U_1,W_1,U_2,W_2\in F$. Since $a_1\leq u_i$, 
\eqref{ksiszeig} yields $U_i=\Clm$, for $i=1,2$. Thus, since $\rmpt{W_1}, \rmpt{W_2}
\in \mathbb R$ are comparable, $u_1$ and $u_2$ are comparable, and $a_1=u_1\wedge u_2\in\set{u_1,u_2}$. This and an analogous argument for $a_2$ show that $\set{a_1,a_2}\subseteq \Mir L$. 

Finally, armed with $\pair{a_1}{a_2}\in \Compp L\cap \Mir L^2$ and using the maximality of $X$, it is straightforward to see that $a_1\parallel a_2$ and that $X\cap\ideal{a_i}$ is a maximal CD-independent set of $\ideal{a_i}$, for $i=1,2$. 
This settles the case $k=2$.

Since the case $k=1$ is evident by the maximality of $X$, we have shown that \eqref{partka} implies \eqref{partkb}.

The implication \eqref{partkb} $\Longrightarrow$ \eqref{partkc} is trivial.

Next, we prove that \eqref{partkc} implies \eqref{partka}. We can assume $k=2$ since the case $k=1$ is trivial. Since $\max(X\setminus\set 1)=\set{a_1,a_2}$, we have $X=\set 1\cup(X\cap \ideal{a_1})\cup(X\cap \ideal{a_2})$, and this is a disjoint union. 
It follows trivially from $\pair{a_1}{a_2}\in\Compp L$ that $X$ is CD-independent. To prove that it is maximal, assume that $u\in L$ such that $X'=X\cup\set u$ is also CD-independent. 
Depending on the ordering on $\set{a_1,u}$, there are three cases.

First, consider the case $u\parallel a_1$. Then the CD-independence of $X'$ gives $a_1\wedge u=0$, and we obtain $u\leq a_2$ from $\pair{a_1}{a_2}\in \Pcompp L$. That is, $u\in X'\cap \ideal{a_2}$. Clearly, $X'\cap \ideal{a_2}$ is CD-independent in $\ideal{a_2}$ and it includes $X\cap\ideal{a_2}$. The maximality of $X\cap\ideal{a_2}$ yields $X'\cap\ideal{a_2}=X\cap\ideal{a_2}$, and we conclude $u\in X'\cap\ideal{a_2}=X\cap\ideal{a_2}\subseteq X$, that is, $u\in X$. 
Second, if we had  $a_1<u<1$, then 
$a_1\not\leq a_2$ would exclude $u\leq a_2$,  $\pair{a_1}{a_2}\in\Pcompp L$ would exclude $u\parallel a_2$, and $u<1=a_1\vee a_2$ would exclude $u\geq a_2$. Thus this case cannot occur.
Third, the case $u\leq a_1$ implies $u\in X$, because $X\cap\ideal{a_1}$ is a maximal CD-independent subset of $\ideal{a_1}$. 
Therefore,  $X'\subseteq X$ and $X$ is a maximal CD-independent subset. This shows that  \eqref{partkc} implies \eqref{partka}, completing the proof of Part~\eqref{partk}.

Now, we deal with Part~\eqref{parth}. Assume that $L$ is a finite distributive lattice and that  \eqref{partha} holds. The first paragraph in the proof of the Main Theorem of Cz\'edli, Hartmann and Schmidt~\cite{czg-h-sch} explicitely says that   $k=|\max(X\setminus\set 1)|$ is at most 2. Hence, using the maximality of $X$,  \eqref{parthb} follows in an obvious way. 

Conversely, assume that \eqref{parthb} holds.  In virtue of \eqref{dstrlcppscp}, 
we conclude the validity of \eqref{partha} by that same argument that proved the implication \eqref{partkc} $\Rightarrow$ 
\eqref{partka}. 
This completes the proof of Theorem~\ref{thmmain}.
\end{proof}

\begin{proof}[Proof of Corollary~\ref{corpsccmirr}] Assume that $a_1,a_2\in L\setminus\set{0,1}$ such that $\pair{a_1}{a_2}$ belongs to $\Compp L\cap \Pcompp L$. For $i\in\set{1,2}$, let $X_i$ be a maximal CD-independent subset of $\ideal{a_i}$, and let $X=X_1\cup X_2\cup \set1$. Clearly, $X$ is a CD-independent subset of $L$ since  $\pair{a_1}{a_2}\in \Compp L$. Hence, we can extend $X$ to a maximal CD-independent subset $X'$ of $L$. 

For the sake of contradiction, suppose $X'\neq X$, and pick an element $u\in X'\setminus X$. Since $u\notin X$, we have $u\not\geq a_1\vee a_2=1$. Hence, $u\not\geq a_1$ or $u\not\geq a_2$. Let, say, $u\not\geq a_1$. If we had $u\leq a_1$, then $X_1\cup\set u$, which is strictly larger than $X_1$, would be a CD-independent subset of $\ideal {a_1}$, contradicting the maximality of $X_1$. Thus $a_1\parallel u$, and the CD-independence of $X'$ yields $a_1\wedge u=0$. This, together with $\pair{a_1}{a_2} \in \Pcompp L$, yields $u\leq a_2$, which clearly contradicts the maximality of $X_2$. It follows that $X=X'$ is a maximal CD-independent subset of $L$. Hence, clearly, $\set{0,1}\cup\atoms L \subseteq X$.
Finally, since the antichain $\set{a_1,a_2}$ equals $\max(X\setminus\set 1)$, Part~\eqref{partk} of Theorem~\ref{thmmain} implies $\set{a_1,a_2}\subseteq \Mir L$.    
\end{proof}

\section{Examples and comments}\label{examplessection}
The proof of  Part \eqref{partk} of Theorem~\ref{thmmain} was based on Proposition~\ref{cirlProp}. Clearly, there exists a purely lattice theoretical proof of Part \eqref{partk} since, in the worst case, we can repeat several parts from the proof of Proposition~\ref{cirlProp}, given in Cz\'edli \cite{czgcircles}. However, the present approach based on circles gives more visual insight and it is much more economic; once we have \cite{czgcircles}, it is natural to use.

The examples given in this section show that the assumptions stipulated in Theorem~\ref{thmmain} are relevant. In fact, we do not see any straightforward way of reasonable generalizations even if $Y$ in Part \eqref{parte} is assumed to be maximal.  Note that it was already proved in Cz\'edli, Hartmann and Schmidt~\cite{czg-h-sch} that distributivity in Proposition~\ref{czghschprop} cannot be replaced by a weaker lattice identity.

\begin{figure}[htc]
\centerline
{\includegraphics[scale=1.0]{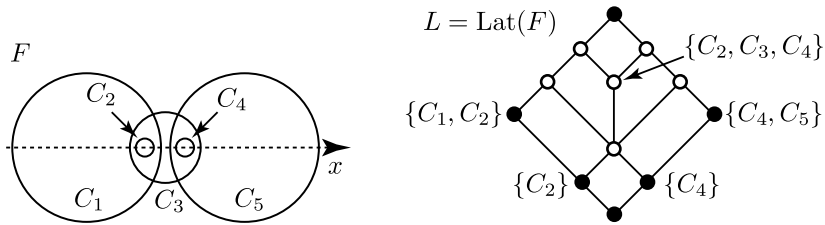}}
\caption{A maximal CD-independent set  in $L$
\label{fig2}}
\end{figure}

\begin{example} By Proposition~\ref{cirlProp}(B), the lattice $L=\lat F$  in Figure~\ref{fig2} is a dually slim and meet-distributive. (It also follows from 
Cz\'edli and Schmidt~\cite[Theorem 12]{czgschvisual} that $L$ has these properties.) We have
$\length L+|\atoms L|=7$. The black-filled elements form a maximal CD-indep\-endent subset of size 6. This shows that in Part~\eqref{parte} of Theorem~\ref{thmmain}, the inequality can be proper even if $L$ is dually slim and $Y$ is a maximal CD-independent subset.
\end{example}

\begin{figure}[htc]
\centerline
{\includegraphics[scale=1.0]{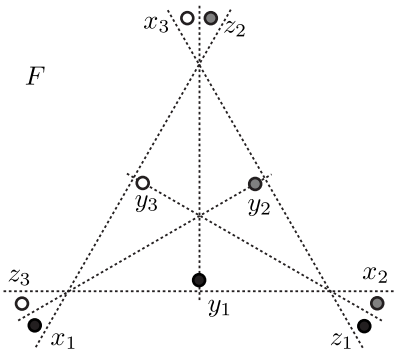}}
\caption{An atomistic example  
\label{fig3}}
\end{figure}

A lattice is \emph{atomistic} if each of its element is the join of some atoms.  The following example indicates that atomicity would not improve Theorem~\ref{thmmain}\eqref{parte}.

\begin{example} Let $L=\lat F$, where $F$ consists of the circles depicted in Figure~\ref{fig3}. The dotted lines indicate how these nine little circles are positioned.
(Note that smaller circles with radius 0, which are points, would also do.)
 Clearly, $L$ is an atomistic lattice, and it is meet-distributive by Proposition~\ref{cirlProp}(A). Since $\atoms L=\bigset{\set x: x\in F}$, we have $|\atoms L|=9$. By an already  mentioned result of \init{S.\,P.\ }Avann~\cite{avann61} and \cite{avann64}, see also \init{M.\ }Stern \cite[Theorem 7.2.27]{stern}, 
$\length L=|\Jir L|=|\atoms L|$. For $i=1,2,3$, let 
$A_i=\set{x_i,z_i}$ and $B_i=\set{x_i,y_i,z_i}$; these subsets of $F$ belong to $L$. Note that $B_1$, $B_2$, and $B_3$ consist of the black-filled circles, the grey-filled circles, and the empty circles, respectively.
It is a straightforward but tedious task to verify that $Y=\atoms L\cup\set{\emptyset, F, A_1, B_1,A_2,B_2,A_3,B_3}$ is a maximal CD-independent subset of $L$. Since $|Y|=17$ and $¡\length L+|\atoms L|=18$, the equality in Part~\eqref{parte} of Theorem~\ref{thmmain} is not an equality in this case.
\end{example}

\section{An application to the theory islands}\label{applicationssection}
The concept of islands appeared first in Cz\'edli~\cite{czgislands}. For definition, let  $m$ and $n$ be natural numbers, and consider an $m$-by-$n$ \emph{rectangular board}, denoted by $\board mn$. It consists of little unite squares called cells, which are arranged in $m$ columns and $n$ rows. For example, $\board 88$ is the chess-board and $\board 84$ is depicted in Figure~\ref{fig4}. Let $h\colon \board mn\to \mathbb R$ be a map, called \emph{height function}. 
A nonempty set $H$ of cells forming a
rectangle is called a \emph{$($cellular$)$ rectangular island} with respect to $h$ if the minimum height of $H$ is greater than the height
of any cell around the perimeter of $H$.
Let us emphasize that the empty set is never a cellular rectangle. 
The concept of islands was motivated by Foldes and Singhi~\cite{foldessinghi}, where cellular rectangular islands on $\board n1$ played a key role in characterizing maximal instantaneous codes.
The number of cellular rectangular islands of the system $\tuple{\board mn;h}$ depends on the height function, and it takes its maximum for some $h$. This maximum value, denoted by $\cfunc mn$, is determined by the following result, where $\lfloor x\rfloor$ stands for the (lower) integer part of $x$. 

\begin{proposition}[\cite{czgislands}] \label{czgislprop}
For $m,n\in\mathbb N$, 
$\cfunc mn= \lfloor (mn+m+n-1)/2\rfloor$.
\end{proposition} 

This result was soon followed by many related ones, due to  Bar\'at, Foldes,   E.\,K.\ Horv\'ath, G.\ Horv\'ath, Lengv\'arszky, 
N\'emeth, Pach, Pluh\'ar, Pongr\'acz, \v{S}e\v{s}elja,  Szab\'o, and Tepav\v cevi\' c. The results of these authors, written alone or in various groups, range from triangular boards to the continuous case and from lattice theory to combinatorics, see 
\cite{barathajnalhorvath}, \cite{foldesatall}, \cite{khemiskolc}, \cite{horvathhorvath}, \cite{horvathnemethpluhar}, \cite{horvathsestep:card}, \cite{horvathsestep:cut},
\cite{lengvarszki}, \cite{pachetall}, \cite{pluhar:brick}, and some further papers not referenced here. 
Since \cite{foldesatall} and \cite{khemiskolc} give  good overviews  on islands, we do not go into further historical details. 
However, we mention the following feature of this research field. At the beginning, in \cite{czgislands} and also in \cite{horvathnemethpluhar} and \cite{pluhar:brick}, a lattice theoretical result of Cz\'edli, Huhn, and Schmidt~\cite{czghuhnsch} on weakly independent subsets played the main role in proofs. Soon afterward, 
simpler approaches were discovered in  \cite{barathajnalhorvath}, and Lattice Theory was more or less neglected thereafter. 

By giving a new proof for Proposition~\ref{czgislprop} based on CD-independence, the goal of this section is to demonstrate that Lattice Theory is still competitive with other approaches. Note that, besides that this was the original motivation in Cz\'edli, Hartmann and Schmidt~\cite{czg-h-sch}, this task was also suggested by Horv\'ath~\cite[Problem 9.1]{khemiskolc}. We only need Proposition~\ref{czghschprop}, taken from \cite{czg-h-sch}, for this purpose.

Each cell of $\board mn$ has exactly four vertices. For a (cellular) rectangular subset $X$ of $\board mn$, let $\grid(X)$ denote the set of vertices of the cells of $X$. We call $\grid(X)$ the \emph{point rectangle} associated with the \emph{cellular rectangle} $X$, while  $\grid(\board mn)$ is the \emph{grid} associated with $\board mn$. In general, a grid is
a  set $\set{0,1,\dots,i}\times \set{0,1,\dots,j}$ of points  for some $i,j\in\mathbb N=\set{1,2,3,\dots}$, shifted to any location in the plane. For a set $\isyst$ of cellular rectangular subsets of $\board mn$, we let $\sgrid(\isyst)=\set{\grid(X): X\in\isyst}$. 
(The letter ``S'' in the mnemonic will remind us ``set''.)
The idea of working with grids rather than boards goes back to  E.\,K.\ Horv\'ath,  G.\ Horv\'ath, N\'emeth, and Szab\'o~\cite{horvathhorvath}.  

First of all, we rephrase Cz\'edli~\cite[Lemma 2]{czgislands}, which was used practically by all previous approaches dealing with (finitely many) islands. The collection of all subsets of a set $U$ is denoted by $\powset\bigl(U)$.

\begin{lemma}[{\cite[Lemma 2]{czgislands}}]\label{trnlstlMa} 
For an arbitrary set $\isyst$ of cellular rectangles of the board $\board mn$, 
the following two conditions are equivalent.
\begin{itemize}
\item
$\isyst$ is the collection of all cellular rectangular islands of $\tuple{\board mn;h}$ for an appropriate height function $h$;
\item $\sgrid(\isyst)$ is a CD-independent subset of  $\bigl\langle\powset\bigl(\grid(\board mn)\bigr),\subseteq\bigr\rangle$ and $\board mn\in \isyst$.
\end{itemize}
\end{lemma}


Note that some authors, including 
Pach, Pluh\'ar,  Pongr\'acz, and Szab\'o~\cite
{pachetall}, call CD-independent subsets as \emph{laminar systems}.

\begin{proof}[Proof of Proposition~\ref{czgislprop}] We do not deal with $\cfunc mn\geq  \lfloor (mn+m+n-1)/2\rfloor$ since this inequality  is proved by an easy construction without any tool, see \cite{czgislands}. 

For brevity, let $G=\powset\bigl(\grid(\board mn)\bigr)$. By Proposition~\ref{czghschprop}, taken from \cite{czg-h-sch}, each maximal CD-independent set of $\tuple{G; \subseteq}$ is of size 
\[\length G+|\atoms G|=  2\cdot |\grid(\board mn) |= 2\cdot {(m+1)(n+1)}\text.
\] 
With the notation $\np=n+1$ and $\mp=m+1$, 
\begin{equation}\label{mpnpsdk}
\text{each maximal CD-independent subset of $G$ is of size $2\mp\np$.}
\end{equation} 
Let $\isyst$ be the collection of all cellular rectangular islands of $\tuple{\board mn; h}$, and denote
$\sgrid(\isyst)$   by $\jsyst$ and $|\jsyst|$ by $t$. Since $|\isyst|=|\jsyst|=t$, it suffices to show the inequality in 
\begin{equation}\label{rRget}
t \leq \lfloor \mp\np/2 \rfloor-1    = \lfloor (mn+m+n-1)/2\rfloor \text.
\end{equation}
We know from Lemma~\ref{trnlstlMa} that $\jsyst$ is a CD-independent subset of $G$. Since the cellular rectangles of $\board mn$ are nonempty by definition, each member of $\jsyst$ consist of at least four points. Therefore the set 
\[\ksyst =\jsyst\cup\set{0_G}\cup\atoms G = \jsyst\cup\set{X: X\subseteq \grid(\board mn)\text{ and }|X|\leq 1 }
\]
is also CD-independent, and it is of size $t+1+ \mp\np$.

A subset $X$ of $\grid(\board mn)$ will be called \emph{bizarre} if $|X|\geq 2$ and there is no rectangle $Y$ of $\board mn$ with $X=\grid (Y)$. We say that a bizarre subset of $\grid(\board mn)$ is \emph{straight}  if all of its points lie on the same vertical or horizontal line. We will only use straight bizarre sets.
We claim that there exist a set $\biz$ of straight bizarre subsets of $\grid(\board mn)$ such that 
\begin{equation}\label{diekZ}
\ksyst\cup\biz\text{ is CD-independent in }G\text{ and }|\biz|\geq 
\begin{cases}t+1&\text{if } 2\mid\mp\np\cr
t+2&\text{if } 2\notmid\mp\np
\end{cases}\text{ .}
\end{equation}
Note that the validity of \eqref{diekZ} will complete the proof as follows. First, let $\mp\np$ be odd. Since $|\ksyst|=t+1+\mp\np$, \eqref{mpnpsdk} and \eqref{diekZ} yield  
$ t+1+\mp\np+t+2\leq |\ksyst|+|\biz|\leq 2\mp\np
$, which clearly implies \eqref{rRget}. For $\mp\np$ even, we conclude \eqref{rRget} from 
$ t+1+\mp\np+t+1\leq |\ksyst|+|\biz|\leq 2\mp\np$ even faster.

We prove \eqref{diekZ} by induction on $mn$. 
Assume that $\mp\np$ is even, and let $U_1,\dots,U_k$ be the list of maximal elements of $\isyst\setminus\set{\board mn}$. 
First, assume $k=1$.
Clearly, at least one of the four sides of $\board mn$ is \emph{separated} from $U_1$ in the sense that no cell on this side belongs to $U_1$. 
Note that $|\isyst\cap\ideal{U_1}|=t-1$, where 
$\ideal U_1=\set{X\in \powset{(\board mn)}: X\subseteq U_1}$ denotes the principal ideal  generated by $U_1$. 
Applying the induction hypothesis  to the subboard $U_1$, we can add at least $(t-1)+1$ straight bizarre subsets of  $\grid(U_1)$ to 
$\ksyst$ to obtain a larger CD-independent subset of $G$. Two neighboring points on the separated side form a straight bizarre set, which we still can add without loosing CD-independence. The set $\biz$ of all these bizarre sets is of size at least $(t-1)+1+1=t+1$, as desired. 

Second, assume $k\geq 2$. 
For $i=1,\dots, k$, let $t_i=|\isyst\cap \ideal {U_i}|$. Clearly, $t=t_1+\dots+t_k+1$. 
By the induction hypothesis, we can add at least $t_i+1$ straight bizarre subsets of $\grid(U_i)$ to $\sgrid(\isyst \cap\ideal{U_i} )$ without spoiling its CD-independence. Since the bizarre subsets we add to $\sgrid(\isyst \cap\ideal{U_i} )$ are disjoint from $\grid(U_j)$ for $j\neq i$, we can add all these bizarre subsets simultaneously to $\ksyst$ without hurting its CD-independence. This way, the set $\biz$ of all straight bizarre subsets we add is at least
\begin{equation}\label{hakkd} (t_1+1)+\dots+(t_k+1)=t+(k-1)\geq t+1\text.
\end{equation}
Hence, \eqref{diekZ} holds in this case again. 

Next, we assume that $\mp\np$ is odd. The treatment of this case is more or less the same as that for $2\mid\mp\np$ but we have to find an appropriate $\biz$ of size at least $t+2$. 
That is, we have to find an extra   straight bizarre subset. 
Hence, it will suffice to compare this case to the  case of $2\mid \mp\np$ wherever it is possible.  Observe that $1\leq m<\mp$ and $2\notmid \mp$ gives $\mp\geq 3$, and we also have $\np\geq 3$. Therefore, if $k=1$, then we can find two comparable straight bizarre subsets of a separated side of $\board mn$ rather than just one, and $|\biz|\geq t+2$ follows the same way  as we obtained $|\biz|\geq t+1$ in the previous argument for $\mp\np$ even and  $k=1$. If $k\geq 3$, then $|\biz|\geq t+2$ comes from \eqref{hakkd}.  

\begin{figure}[htc]
\centerline
{\includegraphics[scale=1.0]{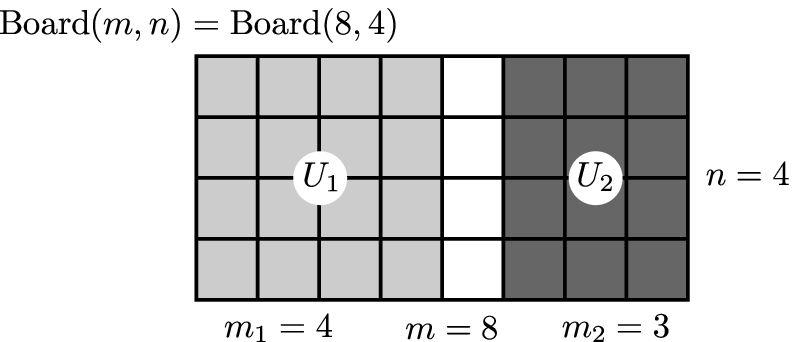}}
\caption{The case of $k=2$ and $2\notmid\mp\np$ 
\label{fig4}}
\end{figure}

Therefore, we are left with the case
$k=2$ such that no side of $\board mn$ is separated from both $U_1$ and $U_2$. 
The situation, up to rotation by ninety degrees, is exemplified in Figure~\ref{fig4}. 
By the maximality of $\isyst$, there is no cellular rectangular subset of $\board mn$ that strictly includes $U_i$ and keeps a positive distance from $U_{3-i}$, for $i\in\set{1,2}$. It follows that three sides of $U_i$ lie on appropriate sides of $\board mn$, and the distance between $U_1$ and $U_2$ is 1. Let $U_i$ be an $m_i$-by-$n$ board for $i\in\set{1,2}$.
Since $\mp=m+1$ is odd and $m=m_1+m_2+1$, one of $\mp_1=m_1+1$ and $\mp_2=m_2+1$ is odd, and the other is even. Let, say, $\mp_1$ be odd. Since $\mp_1\np$ is odd, the induction hypothesis allows us to achieve $t_1+2$ instead of $t_1+1$ in \eqref{hakkd}, and $|\biz|\geq t+2$ follows again.
\end{proof}

\begin{remark} Since we only used straight bizarre subsets rather than arbitrary bizarre ones in the proof above, this method, possibly with non-straight bizarre subsets, will hopefully work for other sorts of boards and islands.
\end{remark}

\begin{ackno} The author is indebted to Eszter K.\ Horv\'ath for her comments on the literature of islands.
\end{ackno}


\begin{thebibliography}{99}

\bibitem{adaricheva}
    K. Adaricheva, 
    Representing finite convex geometries by relatively convex sets,
    \emph{European Journal of Combinatorics}, to appear; {\url{http://arxiv.org/abs/1101.1539}} .

\bibitem{adarichevaczg}
  K. Adaricheva and G. Cz\'edli,
  Notes on the description of join-distributive lattices by permutations,
  \emph{Algebra Universalis}, submitted,
  {\url{http://arxiv.org/abs/1210.3376}} .

\bibitem{r:adarichevaetal}
  K. Adaricheva, V. A. Gorbunov and V.I. Tumanov, V.I.:
  Join-semidistributive lattices and
convex geometries, 
  \emph{Advances in Math.}, \tbf{173} (2003),  1--49. 


\bibitem{avann61}
  S.P. Avann,  
  Application of the join-irreducible excess function to semimodular lattices,
  \emph{Math. Annalen}, \tbf{142}  (1961), 345--354.  


\bibitem{avann64}
  S.P. Avann,  
  Increases in the join-excess function in a lattice, 
   \emph{Math. Ann.}, \tbf{154} (1964), 420--426. 

\bibitem{barathajnalhorvath}
   J. Bar\'at, P. Hajnal and E.K. Horv\'ath,
   Elementary proof techniques for the maximum number of islands,
    \emph{European Journal of Combinatorics}, \tbf{32} (2011),  276--281. 


\bibitem{caspardmonjardet}
   N. Caspard and B. Monjardet, 
   Some lattices of closure systems on a finite set,
   \emph{Discrete Mathematics and Theoretical Computer Science}, \tbf 6 (2004), 163--190.  

\bibitem{czgislands}
   G. Cz\'edli, 
   The number of rectangular islands by means of distributive lattices,
   \emph{European Journal of Combinatorics}, \tbf{30} (2009), 208--215.


\bibitem{czgcoord}
    G. Cz\'edli,  
    Coordinatization of join-distributive lattices,
    \emph{Algebra Universalis}, submitted, {\url{http://arxiv.org/abs/1208.3517}} .

\bibitem{czgcircles}
    G. Cz\'edli,  
    Finite convex geometries of circles,
    \emph{Discrete Mathematics}, submitted,
    {\url{http://arxiv.org/abs/1212.3456}} .

\bibitem{czgggresect}
   G. Cz\'edli and G.  Gr\"atzer, 
   Notes on planar semimodular lattices. VII. Resections of planar semimodular lattices,
    \emph{Order},  {\url{DOI 10.1007/s11083-012-9281-1}}, 
 published online 12  December 2012. 


\bibitem{czg-h-sch}
  G. Cz\'edli, M. Hartmann and E.T. Schmidt, 
  CD-independent subsets in distributive lattices,
  \emph{Publicationes Mathematicae Debrecen}, \tbf{74/1--2} (2009), 127--134.


\bibitem{czghuhnsch}
   G. Cz\'edli, A. P. Huhn and E. T. Schmidt, 
   Weakly independent subsets in lattices,
   \emph{Algebra Universalis}, \tbf{20} (1985), 194--196. 

\bibitem{czgozsvudv}
   G. Cz\'edli, L. Ozsv\'art and B. Udvari:
  How many ways can two composition series intersect?,
   \emph{Discrete Mathematics}, \tbf{312} (2012), 3523--3536.


\bibitem{czgschtJH}
   G. Cz\'edli and E.T. Schmidt, 
   The Jordan-H\"older theorem with uniqueness for groups and semimodular
lattices, 
   \emph{Algebra Universalis},   \tbf{66} (2011), 69--79.


\bibitem{czgschvisual}
   G. Cz\'edli and E.T. Schmidt, 
   Slim semimodular lattices. I. A visual approach,
   \emph{Order}, \tbf{29} (2012), 481--497.


\bibitem{czgschslim2}
   G. Cz\'edli and E.T. Schmidt, 
   Slim semimodular lattices. II. A description by patchwork systems,
   \emph{Order},  \tbf{30} (2013), 689--721.

\bibitem{czgschcompser}
   G. Cz\'edli and E.T. Schmidt, 
   Composition series in groups and the structure of slim semimodular lattices, 
   \emph{Acta Sci. Math. (Szeged)}, to appear, {\url{http://arxiv.org/abs/1208.4749}} .


\bibitem{r:dilworth40}
    R. P. Dilworth,  
    Lattices with unique irreducible decompositions,
    \emph{Annals of Mathematics (2)},  
    \tbf{41} (1940),  771--777.

\bibitem{edelman} 
     P.H. Edelman,
     Meet-distributive lattices and the anti-exchange closure, 
     \emph{Algebra Universalis}, \tbf{10} (1980), 290--299. 

\bibitem{foldesatall}
   S. Foldes, E.K. Horv\'ath, S. Radeleczki and T. Waldhauser,
    A general framework for island systems,
    {\url{http://www.math.u-szeged.hu/\textasciitilde{}horvath/}} .

\bibitem{foldessinghi}
    S. Foldes and N. M. Singhi, 
    On instantaneous codes,
    \emph{J. of Combinatorics, Information and
System Sci.}, \tbf{31} (2006), 317--326.

\bibitem{gratzerknapp} 
   G. Gr\"atzer and E. Knapp, 
   Notes on planar semimodular lattices. I.  Construction, 
   \emph{Acta Sci.\ Math.\ (Szeged)}, \tbf{73} (2007), 445--462.


\bibitem{khemiskolc}
   E. K. Horv\'ath,
   Islands from coding theory to enumerative combinatorics and to lattice theory –-- overview and open problems, 
   {\url{http://www.math.uszeged.hu/\textasciitilde{}horvath/}} .
   
   

\bibitem{horvathhorvath}
   E. K. Horv\'ath,  G. Horv\'ath, Z. N\'emeth and Cs. Szab\'o, 
   The number of square islands on a rectangular sea,
   \emph{Acta Sci. Math. (Szeged)}, \tbf{76} (2010),  35--48.

\bibitem{horvathnemethpluhar}
   E. K. Horv\'ath,  Z. N\'emeth and G. Pluh\'ar, 
   The number of triangular islands on a triangular grid,
   \emph{Periodica Mathematica Hungarica}, \tbf{58} (2009), 25--34.  

\bibitem{horvathradCD}
   E. K. Horv\'ath and S. Radeleczki,
   Notes on CD-independent subsets,
   \emph{Acta Sci. Math. (Szeged)}, \tbf{78} (2012), 3--24.

\bibitem{horvathsestep:card}
   E. K. Horv\'ath, B. \v{S}e\v{s}elja and A. Tepav\v cevi\' c, 
   Cardinality of height function's range in case of maximally many rectangular islands
--- computed by cuts,
    \emph{Central European J. Math.}, \tbf{11(2)} (2013), 296--307.

\bibitem{horvathsestep:cut}
   E. K. Horv\'ath, B. \v{S}e\v{s}elja and A. Tepav\v cevi\' c, 
   Cut approach to islands in rectangular fuzzy relations.
   \emph{Fuzzy Sets and Systems}, \tbf{161} (2010), 3114--3126.


\bibitem{lengvarszki}
  Zs.\ Lengv\'arszky,
  Notes on systems of triangular islands,
  \emph{Acta Sci. Math. (Szeged)}, \tbf{75} (2009), 369--376.

\bibitem{monjardet}
  B. Monjardet, 
  A use for frequently rediscovering a concept,
  \emph{Order},  \tbf{1} (1985), 415--417.


\bibitem{pachetall}
  P. P. Pach, G. Pluh\'ar, A.\ Pongr\'acz and Cs. Szab\'o, 
  The possible number of islands on the sea,
  \emph{J. Math. Anal. Appl.}, \tbf{375}  (2011), 8--13.


\bibitem{pluhar:brick}
   G. Pluh\'ar,
   The number of brick islands by means of distributive lattices,
   \emph{Acta Sci. Math. (Szeged)}, \tbf{75} (2009), 3--11.
   



\bibitem{stern}
  M. Stern, 
  \emph{Semimodular Lattices. Theory and Applications, Encyclopedia of Mathematics and
its Applications}, \tbf{73}, 
  Cambridge University Press (1999).





\end{thebibliography}
\end{document}